\documentclass{amsart}

\newtheorem{theorem}{Theorem}[section]

{\theoremstyle{example}\newtheorem{example}[theorem]{Example}}
{\theoremstyle{remark}\newtheorem{remark}[theorem]{Remark}}

\usepackage{graphicx}
\usepackage[T1]{fontenc}

\usepackage{amssymb,amsmath}

\usepackage{multirow}

\begin{document}
\title[Markoff--Rosenberger triples in geometric progression]{Markoff--Rosenberger triples\\ in geometric progression}
%\title[Markoff--Rosenberger triples in geometric progression]{Markoff--Rosenberger triples in geometric progression over number fields}
%\UniCountry{Universidad Aut\'onoma de Madrid, Spain}
\author{Enrique Gonz\'alez--Jim\'enez}
\address{Universidad Aut{\'o}noma de Madrid, Departamento de Matem{\'a}ticas and Instituto de Ciencias Matem{\'a}ticas (ICMat), Madrid, Spain}
\email{enrique.gonzalez.jimenez@uam.es}
\urladdr{http://www.uam.es/enrique.gonzalez.jimenez}
\thanks{The author was partially  supported by the grant MTM2012--35849.}

%%%%%%%%%%%%%%%%%%%%%%%%%%%%%%%%%%%%%%%%%%%

\date{\today}
\begin{abstract}
We study solutions of the Markoff--Rosenberger equation $ax^2+by^2+cz^2 = dxyz$ whose coordinates belong to the ring of integers of a number field and form a geometric progression.
\end{abstract}

\maketitle

%    Text of article.

\section{Introduction}
The study of sequences of points on algebraic varieties is nowadays a topic in number theory studied by many authors. A special interest has been shown for the case of arithmetic progressions on plane curves. Let $C:F(x,y)=0$ be a plane curve defined over a field $K$. An arithmetic progression (a.p. for short) of length $n$ on $C$ is a sequence of points $(x_1,y_1),\dots,(x_n,y_n)\in C(K)$ such that $x_1,\dots,x_n$ form an arithmetic progression. If $C$ is an elliptic curve, many researchers have studied this problem depending on how $C$ is described: Weierstrass form  \cite{Bremner1999,Campbell2003,Garcia-Selfa-Tornero2005, Garcia-Selfa-Tornero2006, Alvarado2011,Schwartz-Solymosi-Zeeuw2011}, Mordell \cite{Mohanty1975,Lee-Velez1992}, congruent  \cite{Bremner-Silverman-Tzanakis2000,Spearman2011}, quartic form \cite{Ulas2005, MacLeod2006, Alvarado2010}, Edwards \cite{Moody2011}, Huff \cite{Moody2011b}. In genus greater than one the case of $C$ being a hyperelliptic curve has been treated in \cite{Ulas2009,Alvarado2009,Ulas2012}. As for the case of genus zero is concerned:  Pellian equations \cite{Dujella-Petho-Tadic2008, Petho-Ziegler2008,Aguirre-Dujella-Peral2012}; conic section \cite{Alvarado-Goins2012}.

Recently, some authors have considered similar problems  but replacing arithmetic progressions by geometric progressions (g.p. for short): Bérczes and Ziegler \cite{Berczes-Ziegler} and Bremner and Ulas \cite{Bremner-Ulas} on Pell type equations. 

One can take another point of view and consider a hypersurface $S:F(x_1,\dots,x_n)=0$ in $\mathbb A^n$. One may study if the coordinates of a point in $S$, considered as a sequence, satisfy some property. For example, if they form an a.p.. The case of solutions of norm equation such that their coordinates form an a.p. has been deeply studied in \cite{Berczes-Hajdu-Petho2010,Berczes-Petho2004,Berczes-Petho2006,Bazso2007}. Recently, the author and J. M. Tornero \cite{Gonzalez-Jimenez-Tornero2013} have studied the case of triples in a.p. on the Markoff--Rosenberger equation over number fields. That is, if $\mathcal O_K$ denotes the ring of integers of a number field $K$, we studied triples $x,y,z\in \mathcal O_K$ in a.p. such that $ax^2+by^2+cz^2=dxyz$ for some $a,b,c,d\in\mathcal O_K$.  In the current paper, we fix our attention on the case of g.p. instead of a.p. to prove results for g.p. analogous to those obtained in \cite{Gonzalez-Jimenez-Tornero2013}. That is, our main objective is to study the set
$$
\mathcal{GP}_{(a,b,c,d)}(K):=\{\mbox{$\mathcal{O}_{K}$--non-trivial triples in g.p. to $ax^2+by^2+cz^2=dxyz$}\}.
$$
In section \ref{s0}, we will see that $\mathcal{GP}_{(a,b,c,d)}(K)$ is in bijection with a subset of the affine solutions over $\mathcal{O}_K$ of a genus zero curve, denoted by  $\mathcal G_{(a,b,c,d)}$. Then the description of  $\mathcal{GP}_{(a,b,c,d)}(K)$ is translated to the computation of integer points of a genus zero curve. For this task we will use mainly and heavily the research of Poulakis on solving genus zero Diophantine equations, developed in successive papers with Alvanos, Bilu and Voskos  \cite{Alvanos-Bilu-Poulakis2009,Alvanos-Poulakis2011,Poulakis-Voskos2000}. In section \ref{NF}, we present an algorithm based on Alvanos and Poulakis' work \cite{Alvanos-Poulakis2011} that allows us to compute $\mathcal G_{(a,b,c,d)}(\mathcal{O}_K)$. We show how the algorithm described in this section works on some examples not covered by the theoretical results from section \ref{s_results}. 

All the main theoretical results appear at section \ref{s_results}. There, we show that infinitely many integral Markoff--Rosenberger triples in g.p. over a number field $K$ can exist only if $K$ is neither the rational field nor a quadratic imaginary field. When $K$ is the rational field we give an explicit description of this finite set. Finally we fix our attention to our original goal, the study of the generalized Markoff equation: $x^2+y^2+z^2=dxyz$ where $d\in\mathbb Z_{>0}$. We obtain the finite set $\mathcal{GP}_{(1,1,1,d)}(K)$ when $K$ is the rational field or an imaginary quadratic field. Furthermore, for the case when $K$ is a real quadratic field we give a completely explicit description of this set, that could be either empty or infinite. 

%%%%%%%%%%%%%%%%%%%%%%%%%%%%%%%%%%%%%%%%%%%%%%%%%
%%%%%%%%%%%%%%%%%%%%%%%%%%%%%%%%%%%%%%%%%%%%%%%%%

\section{A genus zero curve}\label{s0}

Our starting point is the Markoff equation
\begin{equation}\label{Markoff}
x^2+y^2+z^2=3xyz.
\end{equation}
The integer solutions (so--called Markoff triples) of this Diophantine equation were deeply studied by Markoff in \cite{Markoff1879, Markoff1880} obtaining, among other results, that infinitely many Markoff triples exist. Later on, the Markoff equation has been generalized by several authors. We are going to focus our attention on the one studied by Rosenberger \cite{Rosenberger1979}:
\begin{equation}\label{MR}
ax^2+by^2+cz^2=dxyz.
\end{equation}
We will call this equation the Markoff--Rosenberger equation and its solutions will be called Markoff--Rosenberger triples. Notice that Rosenberger imposed some extra conditions on the coefficients $a,b,c,d\in\mathbb{N}$: $a|d$, $b|d$, $c|d$ and $(a,b)=(a,c)=(b,c)=1$. With these requirements he proved that non--trivial integral Markoff--Rosenberger triples exist only if $(a,b,c,d)$ belongs to the set $\{(1,1,1,1)$, $(1,1,1,3)$, $ (1,1,2,2),(1,1,2,4)$,$(1,2,3,6),(1,1,5,5)\}$. We will not assume these extra conditions for the rest of the paper. 

This paper is devoted to the study of Markoff--Rosenberger triples that form a g.p. over $\mathcal{O}_K$. Let $a,b,c,d\in\mathcal O_K$ and $x,y,z\in \mathcal{O}_K$ be a Markoff--Rosenberger triple in g.p.. Then there exist $\alpha,\beta\in\mathcal{O}_K$ such that 
$$
x=\alpha,\quad y=\alpha\beta, \quad z=\alpha\beta^2, 
$$
satisfying
$$
\alpha^2(c\beta^4-d\alpha\beta^3+b\beta^2+a)=0.
$$
Therefore, if we exclude the solutions with $\alpha=0$ that correspond to the trivial solution $(x,y,z)=(0,0,0)$, we obtain that non-trivial Markoff--Rosenberger triples in g.p. over $\mathcal{O}_K$ are in bijection with the affine solutions over $\mathcal{O}_K$ of the curve 
$$
\mathcal G=\mathcal G_{(a,b,c,d)}\,:\,F(X,Y)=cY^4-dXY^3+bY^2+a=0.
$$
with $X\ne 0$. Namely:
\begin{equation}\label{biye}
\begin{array}{ccc}
\mathcal G_{(a,b,c,d)}(\mathcal O_K) &\longrightarrow &\mathcal{GP}_{(a,b,c,d)}(K).\\[1.5mm]
(X,Y)&\longmapsto &(X,XY, XY^2)\\[0.2mm]
(x,z/y)&\longleftarrow{{\hspace{-5pt}\shortmid}} &(x,y,z)\\
\end{array}
\end{equation}
Let $\widetilde{F}(X,Y,Z)$ be the homogenization of $F(X,Y)$ and denote by $\widetilde{\mathcal G}$ the projective curve defined by $\widetilde{\mathcal G}\,:\,\widetilde{F}(X,Y,Z)=0$. Now, the curve $\widetilde{\mathcal G}$ has two points at infinity: $[c\,:\,d\,:\,0]$ and the singular point $[1:0:0]$. In particular, $\widetilde{\mathcal G}$ has genus $0$ and the birational map
\begin{equation}\label{para}
\begin{array}{cccl}
\widetilde{\psi}:&\mathbb P^1&\longrightarrow &\widetilde{\mathcal G}\\
             & [U:V] & \longmapsto & [cU^4+bV^2U^2+aV^4 \,:\, dU^4\,:\, dU^3V]\\
\end{array}
\end{equation}
gives a parametrization of the curve $\widetilde{\mathcal G}$ over $K$.

Denote by $\widetilde{\mathcal G}_\infty$ the set of infinity places of the field $\overline{K}(C)$. Then by \cite[Lemma 2.2]{Poulakis-Voskos2000} we have that 
$$
|\widetilde{\mathcal G}_\infty|=|\{[U:V]\in \mathbb P^1\,:\, dU^3V=0\}|=|\{[1:0],[0:1]\}|=2.
$$
Note that both points at infinity are defined over $\mathbb Q$.

%%%%%%%%%%%%%%%%%%%%%%%%%%%%%%%%%%%%%%%%%%%%%%%%%
%%%%%%%%%%%%%%%%%%%%%%%%%%%%%%%%%%%%%%%%%%%%%%%%%
\section{The general algorithm over number fields}\label{NF}
Let $K$ be a number field of degree $n=[K:\mathbb Q]$ and $a,b,c,d\in\mathcal O_K$. Our objective in this section is to describe the set 
$$
\mathcal G(\mathcal{O}_K)=\{(X,Y)\in{\mathcal{O}_{K}}^2\,|\, F(X,Y)=0\}.
$$
For this purpose we use the affine part of the parametrization $\widetilde{\psi}$:
\begin{equation}\label{parafin}
\begin{array}{cccl}
{\psi}:&\mathbb P^1&\longrightarrow &{\mathcal G}\\
             & [U:V] & \longmapsto & \displaystyle \left(\frac{cU^4+bV^2U^2+aV^4}{dU^3V},\frac{U}{V}\right)\\
\end{array}
\end{equation}
and we are going to develop an algorithm heavily based on the Alvanos and Poulakis' algorithm \texttt{INTEGRAL-POINTS2A} from the paper \cite{Alvanos-Poulakis2011}. In fact, the algorithm presented here is just the application of \texttt{INTEGRAL-POINTS2A} to the curve $\mathcal G$. 

\noindent {\bf Step 1}:  Let $\alpha\in\mathcal O_K$ and denote by $\widehat{\alpha}=\alpha^s$ where $s=0$ or $s=1$ depending on whether $\alpha\in\mathbb Z$ or $\alpha\notin\mathbb Z$. Let $E=\mathbb Q(\widehat{c})$ and $L=\mathbb Q(\widehat{d}\,)$. Compute
 $$
\delta_0=gcd\left(\frac{c}{\widehat{c}}\,\mathcal N_{E}(\widehat{c}), \frac{d}{\widehat{d}}\,\mathcal N_{L}(\widehat{d}\,)\right)\in\mathbb Z.
 $$ 
We denote by $\mathcal{N}_{K}$ the absolute norm map for a number field $K$.

\noindent {\bf Step 2}:  Compute 
$$
M=\left\{t\in\mathcal O_K\,|\, \mbox{$\mathcal N_K(t)$  divides $\mathcal N_K(a)\delta_0^{4n}$}\right\}_{/ \sim},
$$
where $\sim$ denotes equivalence class (modulo units of $\mathcal O_K$).

\noindent {\bf Step 3}: Compute a basis for the unit group $\mathcal U(\mathcal O_K)$. By Dirichlet's Unit Theo\-rem we have that $\mathcal U(\mathcal O_K)$ is a finitely generated abelian group and therefore:
$$
\mathcal U(\mathcal O_K)=\langle\zeta_k \rangle\oplus\langle \varepsilon_1,\dots,\varepsilon_r\rangle
$$
where $\zeta_k$ is a $k$-th root of unity ($k$ is the torsion order) and $r$ is the rank. Moreover, $r=r_1+r_2-1$ where $r_1$ denotes the number of real embeddings of $K$ and $r_2$ the number of complex pairs embeddings of $K$. In particular, $\mathcal U(\mathcal O_K)$ is finite if and only if $K=\mathbb Q$ or $K$ is an imaginary quadratic field.

\noindent {\bf Step 4}: For any $t\in M$ compute $\tau(i,t)$ the order of the class of $\varepsilon_i$ in $\mathcal U (\mathcal O_K/\delta_0dt^3)$. Note that if $\delta_0dt^3\in \mathcal U(\mathcal O_K)$ then we define $\tau(i,t)=1$.

\noindent {\bf Step 5}: Now, for every $t\in M$ compute the set $H(t)$ of units $\eta=\zeta_k^l\varepsilon_1^{l_1}\cdots\varepsilon_r^{l_r}$ with $0\le l<k,0\le l_i<\tau(i,t)$ for $i=1,\dots,r$ such that $\psi(t\eta,\delta_0)\in \mathcal O_K^2$. From this condition we obtain that $d(t\eta)^3\delta_0$ divides $c(t\eta)^4+b\delta_0^2(t\eta)^2+a\delta_0^4$ and $\delta_0$ divides $t\eta$. Then
$$
H(t)=\left\{\eta=\zeta_k^l\varepsilon_1^{l_1}\cdots\varepsilon_r^{l_r}\,\Big|\,
\begin{array}{c}
0\le l<k,\quad 0\le l_i<\tau(i,t),\,\,\,i=1,\dots,r\\[1mm]
\mbox{$d(t\eta)^3\delta_0|(c(t\eta)^4+b\delta_0^2(t\eta)^2+a\delta_0^4)$ and $\delta_0|t\eta$}
\end{array}
\right\}.
$$
Notice that in the case that $\delta_0\ne\pm 1$ we have $\delta_0$ divides $t$. 

\noindent {\bf OUTPUT}: Denote by $\Theta(t)=\left\{\prod_{i=1}^r \varepsilon_i^{\tau(i,t)z_i}\,\Big|\,z_i\in\mathbb Z\,,i=1,\dots,r\right\}$. Then 
$$
\displaystyle \mathcal G(\mathcal O_K)=\bigcup_{t\in M}\{\psi(t\eta\varepsilon,\delta_0)\,|\, \eta\in H(t)\,\,\,\mbox{and}\,\,\, \varepsilon\in\Theta(t)\}.
$$

Note that in order to apply the algorithm for some fixed values $a,b,c,d$ in the ring of integers of some number field $K$ we should be able to solve some problems. All of them are sorted out in \verb|Magma| \cite{magma2.18-8}. In the following table we show the main problems to be solved and the \verb|Magma| functions that may be used:\\
\begin{center}
\begin{tabular}{|c|c|c|}
\hline 
Step & Problems & \verb|Magma| functions\\
\hline
2 & $M$ & \verb|NormEquation|\\
\hline
3 & $\mathcal U(\mathcal O_K)$ & \verb|UnitGroup|\\
\hline
4 & $\tau(i,t)$ & \verb|quo|, \verb|MultiplicativeGroup|, \verb|Order| \\
\hline
5 & $\psi(t\eta,\delta_0)\in \mathcal O_K^2$ & \verb|IsIntegral| \\
\hline
\end{tabular}
\end{center}

\subsection{The algorithm at work}
We recall that Rosenberger \cite{Rosenberger1979} proved that  integral Markoff--Rosenberger triples exist if and only if 
$$
(a,b,c,d) \in \left\{ (1,1,1,1), \; (1,1,1,3), \; (1,1,2,2), \; (1,1,2,4), \; (1,2,3,6), \; (1,1,5,5) \right\},
$$
with the extra requirements: $a,b,c,d\in\mathbb{N}$: $a|d$, $b|d$, $c|d$ and $(a,b)=(a,c)=(b,c)=1$. Note that Theorem \ref{TZ} applied to the previous cases tells us that 
\begin{center}
\begin{tabular}{|c|c|}
\hline 
$(a,b,c,d)$ & $\mathcal{GP}_{(a,b,c,d)}(\mathbb Q)$\\
\hline
$(1,1,1,1)$ & $(\pm 3,3,\pm 3)$\\
\hline
$(1,1,1,3)$ & $(\pm 1, 1,\pm 1)$\\
\hline
$(1,1,2,2)$ & $(\pm 2,2,\pm 2)$\\
\hline
$(1,1,2,4)$ & $(\pm 1,1,\pm 1)$\\
\hline
$(1,2,3,6)$ & $(\pm 1, 1,\pm 1)$\\
\hline
$(1,1,5,5)$ & $\emptyset$\\
\hline
\end{tabular}
\end{center}
That is, integral Markoff--Rosenberger triples in g.p. exist for all these cases except for the last one. In this section we are going to study the case $(1,1,5,5)$ over the first two quadratic real fields. Thanks to Theorem \ref{teo} we know that if $D$ is a squarefree positive integer then a Markoff--Rosenberger triple in g.p. over $\mathbb Q(\sqrt D)$ exists if and only if infinitely many exist. We show one case of each of these possibilities. In particular, in $\mathbb Q(\sqrt 2)$ we describe the infinitely many triples in g.p.; meanwhile in $\mathbb Q(\sqrt 3)$ we will show that triples in g.p. do not exist. Let us apply our algorithm for this purpose. 

Let $K$ be a real quadratic field. We have that $\delta_0=5$, in particular $M=\left\{t\in\mathcal O_K\,|\, \mathcal N_{K}(t)=\pm 5^{k},\,k=2,\dots,8\right\}_{/ \sim}$. Moreover, a fundamental unit $\varepsilon_D$ exists, such that $\mathcal U(\mathcal O_K)=\{\pm \varepsilon_D^k\,|\, k\in\mathbb Z\}$. Now we work out the cases $K=\mathbb Q(\sqrt 2)$ and $K=\mathbb Q(\sqrt 3)$:

\noindent $\bullet$ {$K=\mathbb Q(\sqrt 2)$}. The fundamental unit is $\varepsilon_2=1+\sqrt{2}$ which has order $12$ on $\mathcal{U}(\mathcal O_K/5)$. We have $M=\bigcup_{k=1}^4\{5^k,5^k\varepsilon_2 \}$. Then, for any $t\in M$ we obtain that $H(t)=\emptyset$, except in the following two cases:
$$
H(5\varepsilon_2)=\{\pm\varepsilon_2^2,\pm\varepsilon_2^8\} \quad\mbox{and}\quad H(5)=\{\pm\varepsilon_2^3,\pm\varepsilon_2^9\}.
$$
Now, we have $\Theta(t)=\{\varepsilon_2^{12z}\,|\, z\in\mathbb Z\}$ for any $t\in M$. Therefore
$$
\displaystyle \mathcal{G}_{(1,1,5,5)}(\mathcal{O}_{\mathbb Q(\sqrt 2)})=\bigcup_{k\in \{1,2\}} \bigcup_{z\in\mathbb Z}\{\psi(\pm 5\varepsilon_2^{3^k+12z},5)\}.
$$
Then we have obtained infinitely many triples in g.p., described by
$$
\mathcal{GP}_{(1,1,5,5)}(\mathbb Q(\sqrt 2))=\bigcup_{k=1}^2 \bigcup_{z\in\mathbb Z}\left\{(\pm\beta,\beta \gamma,\pm\beta\gamma^2 )\,\left|\,
\begin{array}{l}\gamma=\varepsilon_2^{3^k+12z},\,\\ \beta=\gamma+5^{-1}(\gamma^{-1}+\gamma^{-3})\end{array}\right.\right\}.
$$

\noindent $\bullet$ {$K=\mathbb Q(\sqrt 3)$}. The fundamental unit is $\varepsilon_3=2+\sqrt{3}$ which has order $3$ on $\mathcal{U}(\mathcal O_K/5)$. In this case, we have $M=\{5^k\,|\,k=1,2,3,4 \}$. Finally, we obtain that $H(t)=\emptyset$ for all $t\in M$. That is, $\mathcal{G}_{(1,1,5,5)}(\mathcal{O}_{\mathbb Q(\sqrt 3)})=\emptyset$, and therefore
$$
\mathcal{GP}_{(1,1,1,5)}(\mathbb Q(\sqrt 3))=\emptyset.
$$

%%%%%%%%%%%%%%%%%%%%%%%%%%%%%%%%%%%%%%%%%%%%%%%%%
%%%%%%%%%%%%%%%%%%%%%%%%%%%%%%%%%%%%%%%%%%%%%%%%%

\section{Theoretical results}\label{s_results}
This section is dedicated to show the theoretical results obtained on Markoff--Rosenberger triples in g.p. over number fields. 
\begin{theorem}\label{teo}
Let $K$ be a number field and $a,b,c,d\in\mathcal O_K$. Assume that $\mathcal G(\mathcal{O}_K)$ contains a non--singular point, then
the set $\mathcal{GP}_{(a,b,c,d)}(K)$ is finite if and only if $K=\mathbb{Q}$ or $K$ is an imaginary quadratic field.
\end{theorem}
\begin{proof}
Let $C$ be an affine algebraic curve of genus $g$ defined over a number field $K$ and denote by $C_\infty$ the set of infinity places of the field $\overline{K}(C)$. Siegel \cite{Siegel1929} proved that if $g>0$ or $|C_\infty |>2$ then $C(\mathcal O_K)$ is finite. However, $C(\mathcal O_K)$ may be finite if $g=0$ and $|C_\infty|\le 2$. These last cases where treated by Alvanos, Bilu and Poulakis \cite{Alvanos-Bilu-Poulakis2009}, obtaining a complete characterization of the cases in which $C(\mathcal O_K)$ is finite. In particular, our genus zero curve $\mathcal G$ satisfies $|\mathcal G_\infty|=2$ and both points at infinity are defined over $\mathbb Q$, therefore \cite[Theorem 1.2]{Alvanos-Bilu-Poulakis2009} asserts that $\mathcal G(\mathcal O_K)$ is finite if and only if $K=\mathbb{Q}$ or $K$ is an imaginary quadratic field.
\end{proof}

\begin{remark}
Assume that $d$ divides $a+b+c$ on $\mathcal O_K$ where $K$ neither is $\mathbb{Q}$ nor an imaginary quadratic field, then $\#\mathcal{GP}_{(a,b,c,d)}(K)=\infty$.
\end{remark}

Now we fix our attention on the case of the rational field or an imaginary\footnote{Baer and Rosenberger \cite{Baer-Rosenberger1998} studied the set of solutions over the ring of integers of an imaginary quadratic field of the equation  $ax^2+by^2+cz^2=dxyz$ with the extra requirements $a,b,c\,|\,d$.} quadratic field, since in these cases we have proved that only a finite number of triples in g.p. exist. In the rational case we have the following result:

\begin{theorem}\label{TZ}
Let $a,b,c,d\in\mathbb Z$. For any $u\in\mathbb Z$ denote by $\gamma_u=\frac{cu^4+bu^2+a}{du^3}$ then
$$
\mathcal{GP}_{(a,b,c,d)}(\mathbb Q)=\left\{\left(\pm\gamma_u,u\gamma_u,\pm u^2\gamma_u\right)\,\Big|\, u^2|a,\,u>0,\,du^3|(cu^4+bu^2+a)\right\}.
$$
Furthermore, if $a$ is squarefree and $\gamma=\frac{a+b+c}{d}$, then 
$$
\mathcal{GP}_{(a,b,c,d)}(\mathbb Q)=
\left\{
\begin{array}{ccl}
\{(\pm\gamma,\gamma,\pm\gamma)\} & & \mbox{if $d|(a+b+c)$},\\[2mm]
\emptyset & & \mbox{otherwise}.
\end{array}
\right.
$$
\end{theorem}
\begin{proof}
Let $x,y,z\in \mathbb Z$ be a triple in g.p. to the equation $ax^2+by^2+cz^2=dxyz$. That is, there exist $u,v\in \mathbb Z$ such that $v\ne 0$ and $x=v, y=vu, z=vu^2$. Then $av^2+bu^2v^2+cv^2u^4=dv^3u^3$. Dividing by $v^2$ we obtain the simpler equation: $a+bu^2+cu^4-dvu^3=0$. This implies $u^2|a$ and $v=\gamma_u$. In particular, since $v\in\mathbb Z$ we have $du^3|(cu^4+bu^2+a)$. Finally, we may assume $u>0$ since $\gamma_{-u}=-\gamma_u$ and therefore if $(\gamma_u,u\gamma_u,u^2\gamma_u)\in \mathcal{GP}_{(a,b,c,d)}(\mathbb Q)$ then $(-\gamma_u,u\gamma_u,-u^2\gamma_u)\in \mathcal{GP}_{(a,b,c,d)}(\mathbb Q)$.

Now, from the condition $u^2|a$, it is clear that if $a$ is squarefree then $u=\pm 1$ must hold.
\end{proof}

\begin{remark}
Notice that the condition on the squarefreeness of $a$ is necessary. For example, $\mathcal{GP}_{(4,1,1,1)}(\mathbb Q)=\{(\pm 6, 6, \pm 6), (\pm 3, 6, \pm 12)\}$.
\end{remark}

The next two results give complete descriptions on the case of generalized Markoff triples in g.p. over number fields of very low degree, that is one or two. Notice that $(x,y,z)$ is a triple in g.p. to the equation $x^2+y^2+z^2=dxyz$ if and only if $(x,-y,z)$ is a triple in g.p. to the equation $x^2+y^2+z^2=-dxyz$. Then we will assume that $d$ is positive. Now, the first result describes precisely the cases in which only finitely many triples in g.p. exist. 

\begin{theorem}   Let $d,D\in\mathbb{Z}_{>0}$, with $D$ squarefree. If $(d,D)\ne (1,1)$, then 
$$
\mathcal{GP}_{(1,1,1,d)}(\mathbb{Q}(\sqrt{-D}))=\mathcal{GP}_{(1,1,1,d)}(\mathbb{Q})=\left\{
\begin{array}{ccl}
\{(\pm 3,3,\pm 3)\} & & \mbox{if $d=1$},\\
\{(\pm 1,1,\pm 1)\} & & \mbox{if $d=3$},\\
\emptyset & & \mbox{if $d\ne 1,3$},
\end{array}
\right.
$$
and $\mathcal{GP}_{(1,1,1,1)}(\mathbb{Q}(i))=\{(\pm 3,3,\pm 3),(\pm i,-1,\mp i)\}$.
\end{theorem}
\begin{proof} 
First notice that the second equality, for $(d,D)\ne (1,1)$, is a consequence of Theorem \ref{TZ} with $a=b=c=1$. To prove the first equality we use the algorithm of section \ref{NF}. We have that $\delta_0=1$. Let  $K=\mathbb{Q}(\sqrt{-D})$ then $M=\{1\}$, since on an imaginary quadratic field all elements have positive norm. Now we should compute the unit group of $\mathcal O_K$. But it is well--known that $\mathcal U(\mathcal O_K)=\langle\zeta_k\rangle,
$ when $k=2$ ($\zeta_2=-1$) if $D\ne 1,3$, $k=4$ ($\zeta_4=i$) if $D=1$ and $k=6$ ($\zeta_6=(1+\sqrt{-3})/2$) if $D=3$. Next step is to compute the set $H(1)$: in this case its elements are $\zeta_k^l$ for $0\le l<k$ such that $d|(\zeta_k^{4l}+\zeta_k^{2l}+1)$. It is a straightforward computation to determine $H(1)$ depending on $D$ and $d$:
$$
H(1)=\left\{
\begin{array}{ccl}
\emptyset & & \mbox{if $D\ne 3$ and $d\ne 1,3$},\\
\{\pm 1\} & & \mbox{if $D\ne 1,3$ and $d=1,3$},\\
\{\zeta_6^l\,|\,l=1,2,4,5\} & & \mbox{if $D=3$ and $d\ne 1,3$},\\
\{\zeta_6^l\,|\,0\le l<6\}& & \mbox{if $D=3$ and $d=1,3$},\\
\{\zeta_4^l\,|\,0\le l<4\} & & \mbox{if $D=1$ and $d=1$},\\
\{\pm 1\} & & \mbox{if $D=1$ and $d=3$}.\\
\end{array}
\right.
$$
Then the algorithm outputs:
$$
\mathcal G(\mathcal O_K)=\left\{\left(\frac{\eta^4+\eta^2+1}{d\eta^3},\eta\right)\,\Big|\, \eta\in H(1)\right\}.
$$
That is
$$
\mathcal G(\mathcal O_K)=\left\{
\begin{array}{ccl}
\emptyset & & \mbox{if $D\ne 3$ and $d\ne 1,3$},\\
\{\pm (3,1)\} & & \mbox{if $D\ne 1,3$ and $d=1$},\\
\{\pm (1,1)\} & & \mbox{if $D\ne 3$ and $d=3$},\\
\{(0,\zeta_6^l)\,|\,l=1,2,4,5\} & & \mbox{if $D=3$ and $d\ne 1,3$},\\
\{(0,\zeta_6^l)\,|\,l=1,2,4,5\}\cup\{\pm (3,1)\}& & \mbox{if $D=3$ and $d=1$},\\
\{(0,\zeta_6^l)\,|\,l=1,2,4,5\}\cup\{\pm (1,1)\}& & \mbox{if $D=3$ and $d=3$},\\
\{\pm (3,1),\pm(i,i)\} & & \mbox{if $D=1$ and $d=1$}.
\end{array}
\right.
$$
Therefore, the bijection given by (\ref{biye}) gives:
$$
\mathcal{GP}_{(1,1,1,d)}(\mathbb{Q}(\sqrt{-D}))=\left\{
\begin{array}{cl}
\emptyset & \mbox{if $d\ne 1,3$},\\
\{(\pm 1,1,\pm 1)\} & \mbox{if $d=3$},\\
\{(\pm 3,3,\pm 3)\} & \mbox{if $D\ne 1$ and $d=1$},\\
\{(\pm 3,3,\pm 3),(\pm i,-1,\mp i)\} & \mbox{if $D=1$ and $d=1$}.
\end{array}
\right.
%\vspace{-0.21cm}
$$
\end{proof}

\begin{remark} In particular the previous result proves:
$$
\bigcup_{d,D\in\mathbb Z_{>0}}\mathcal{GP}_{(1,1,1,d)}(\mathbb{Q}(\sqrt{-D}))=\left\{
(\pm 1,1,\pm 1),(\pm 3,3,\pm 3),,(\pm i,-1,\mp i)\}\right.
$$
\end{remark}

\begin{remark} Note that a similar study for fixed $a,b,c\in\mathbb Z$, with $a$ squarefree, could be done. That is, to compute explicitly $\mathcal{GP}_{(a,b,c,d)}(\mathbb{Q}(\sqrt{-D}))$ and $\mathcal{GP}_{(a,b,c,d)}(\mathbb{Q})$ for any $d,D\in\mathbb{Z}_{>0}$, with $D$ squarefree.
\end{remark}
Once we have treated the cases where there are only a finite number of generalized Markoff triples in g.p., we are going to give a complete description of the set of generalized Markoff triples in g.p. over a real quadratic field. In this case, we will have infinitely many such triples or none at all.
\begin{theorem}
Let $D$ be a squarefree positive integer and $\varepsilon_D$ be the fundamental unit of the real quadratic field $K=\mathbb Q(\sqrt{D})$. Denote by $n$ the order of the class of $\varepsilon_D$ in $\mathcal U (\mathcal O_K/d)$ and define the set
$$
\mathcal H=\left\{k\in\{0,\dots, n\}\,|\,\mbox{$\varepsilon_D^{4k}+\varepsilon_D^{2k}+1\equiv 0\,(\mbox{mod$(d\,\mathcal O_K)$}) $}\right\}.
$$
Then
$$
\displaystyle \mathcal{GP}_{(1,1,1,d)}(\mathbb Q(\sqrt{D}))=\!\bigcup_{k\in\mathcal H} \bigcup_{z\in\mathbb Z}\!\left\{(\pm\beta,\beta \gamma,\pm\beta\gamma^2 )\,\Big|\,\gamma=\varepsilon_D^{nz+k},\,\beta=\frac{1}{d}\sum_{j=0}^2{\gamma^{1-2j}}\right\}.
$$
\end{theorem}
\begin{proof}
To prove this result we are going to use the algorithm of section \ref{NF}. First of all, let $\varepsilon_D$ be the fundamental unit of the real quadratic field $K=\mathbb Q(\sqrt{D})$. That is, $\varepsilon_D$ satisfies $\mathcal U(\mathcal O_K)=\{\pm \varepsilon_D^k\,|\,k\in\mathbb Z\}$. To compute $M$ we have only to observe that since $\delta_0=1$ the set $M$ consists on $1$ and $\varepsilon_D$ in the case that $\mathcal N_K(\varepsilon_D)=-1$. Note that this is the so-called negative Pell equation. While the Pell equation $\mathcal N_K(\varepsilon_D)=1$ has always a solution for any positive integer $D$, the negative Pell equation does not always have one, although there exists an effective algorithm to determine, for a fixed $D$, which is the case. 

Now, denote by $n$ the order of the class of $\varepsilon_D$ in $\mathcal U (\mathcal O_K/d)$. Then at step $4$ we must compute $H(t)$ for $t\in M$. Note that $1$ always belongs to $M$, then we have
$$
H(1)=\left\{\pm \varepsilon_D^k\,|\,\mbox{$\varepsilon_D^{4k}+\varepsilon_D^{2k}+1\equiv 0\,(\mbox{mod$(d\mathcal O_K)$}) $},\, k=0,\dots,n-1\right\}.
$$
In the case that $\mathcal N_K(\varepsilon_D)=-1$ we have that $\varepsilon_D\in M$ and
$$H(\varepsilon_D)=\left\{\pm \varepsilon_D^{k}\,|\,\mbox{$\varepsilon_D^{4(k+1)}+\varepsilon_D^{2(k+1)}+1\equiv 0\,(\mbox{mod$(d\mathcal O_K)$}) $},\, k=0,\dots,n-1\right\}.
$$
Finally, we have $\Theta(t)=\left\{ \varepsilon_D^{n z}\,\Big|\,z\in\mathbb Z\right\}$ for any $t\in M$. Then the algorithm outputs
$$
\displaystyle \mathcal {G}_{(1,1,1,d)}(\mathcal O_K)=\{\psi(\pm\varepsilon_D^{k+nz},1)\,|\, k\in \mathcal H,\,z\in\mathbb Z\}.
$$
Then the bijection given by (\ref{biye}) together with the parametrization $\psi$ given in (\ref{parafin}) give the result.
\end{proof}
\begin{remark}
The triples in g.p. over the ring of integer of a real quadratic field of the generalized Markoff equation $x^2+y^2+z^2=dxyz$ have a very special shape. That is, let $(\beta,\gamma\beta,\gamma^2\beta)$ be such a triple. Then $\gamma=\varepsilon_D^{nz+k}$ and $\beta=\frac{1}{d}\sum_{j=0}^2{\gamma^{1-2j}}$ for some $n,k, z\in\mathbb Z$, where $\varepsilon_D$ is the fundamental unit on the real quadratic field $\mathbb Q(\sqrt D)$. Moreover, $\gamma\beta\in\mathbb Z$ since $\gamma\beta=\frac{1}{d}(1+\mathcal{T}race_K(\gamma^2))\in\mathcal O_K$ and the trace of any $\alpha\in \mathcal O_K$ belongs to $\mathbb Z$. But there is something more. The third term of the g.p. satisfies $\gamma^2\beta=\mathcal N_K(\gamma) \beta^\sigma=\pm \beta^\sigma$ where $\sigma$ is the nontrivial automorphism of the Galois group of $\mathbb Q(\sqrt D)$ and the sign depends on $\mathcal N_K(\varepsilon_D)$ and on $n,k, z\in\mathbb Z$.
\end{remark}

\begin{example} We are looking for all Markoff triples in g.p. over $\mathbb Q(\sqrt{5})$. First, the fundamental unit is $\varepsilon_5=(1+\sqrt{5})/2$ whose class in $\mathcal U(\mathcal O_{\mathbb{Q}(\sqrt 5)}/3)$  has order $n=8$. Next, we obtain $\mathcal H=\{0,4,8\}$. Therefore, we conclude:
$$
\displaystyle \mathcal G_{(1,1,1,3)}(\mathbb Q(\sqrt{5}))= \bigcup_{z\in\mathbb Z}\left\{(\pm\beta,\beta \gamma,\pm\beta\gamma^2 )\,\left|\,\gamma=\varepsilon_5^{4z},\,\beta=\frac{1}{3}\sum_{j=0}^2{\gamma^{1-2j}}\right.\right\}.
$$
The following table shows some examples of Markoff triples in g.p. over $\mathbb Q(\sqrt{5})$:\\
\begin{center}
\begin{tabular}{|c|c|}
\hline 
$z$ & Markoff triples in g.p. over $\mathbb Q(\sqrt{5})$\\
\hline
$0$ & $( 1,1,1)$\\
\hline
$1$ & $(56-24\sqrt{5} , 16,56+24\sqrt{5} )$\\
\hline
$2$ & $(17296-7728\sqrt{5} , 736,  17296+7728\sqrt{5})$\\
\hline
$3$ & $(5564321-2488392\sqrt{5}, 34561, 5564321+2488392\sqrt{5})$\\
\hline
$4$ & $(1791660256-801254496\sqrt{5}, 1623616, 1791660256+801254496\sqrt{5})$\\
\hline
\end{tabular}
\end{center}

\end{example}
Finally, we present a partial result valid for any number field.
\begin{theorem} Let $K$ be a number field and $a,b,c,d\in\mathcal{O}_K$ such that  $a,d\in\mathcal U(\mathcal{O}_K)$. For any $u\in\mathcal{U}(\mathcal{O}_K)$ denote by $\gamma_u=d^{-1}(cu+bu^{-1}+au^{-3})$. Then
$$
\mathcal{GP}(a,b,c,d,K)=\{(\gamma_u,u \gamma_u, u^2 \gamma_u) \,|\, u\in\mathcal{U}(\mathcal{O}_K)\}.
$$
\end{theorem}
\begin{proof} 
Let us apply the algorithm of section \ref{NF}.  Now, since $a,d\in\mathcal U(\mathcal{O}_K)$ we have $\delta_0=1$, $M=\left\{t\in\mathcal O_K\,|\, \mathcal N_K(t)=\pm 1\right\}_{/ \sim}$ and $\tau(i,t)=1$ for $i=1,\dots,r=\mbox{rank}_{\mathbb Z}\,\mathcal U(\mathcal O_K)$ and any $t\in M$. Therefore, if $t\in M$ we have that $H(t)=\mathcal U(\mathcal O_K)_{tors}$ and $\Theta(t)=\mathcal U(\mathcal O_K)_{free}$, the torsion and free part of the unit group $\mathcal U(\mathcal O_K)$ respectively. Then the algorithm outputs
$$
\displaystyle \mathcal G(\mathcal O_K)=\bigcup_{t\in M}\{\psi(t\eta\varepsilon,1)\,|\, \eta\in \mathcal U(\mathcal O_K)_{tors}\,,\, \varepsilon\in\mathcal U(\mathcal O_K)_{free}\}=\!\!\!\!\bigcup_{u\in \mathcal U(\mathcal O_K)}\!\!\{\psi(u,1)\}.
$$
Then we conclude the proof using (\ref{biye}) and (\ref{parafin}).
\end{proof}

%%%%%%%%%%%%%%%%%%%%%%%%%%%%%%%%%%%%%%%%%%%%%%%%%
%%%%%%%%%%%%%%%%%%%%%%%%%%%%%%%%%%%%%%%%%%%%%%%%%
	
{\bf Acknowledgements.}  We would like to thank Paraskevas Alvanos and Dimi\-trios Poulakis for several useful discussions on their papers; and José M. Tornero and David Torres Teigell for their careful readings of a preliminary version of this manuscript. 

We wish to thank the anonymous referee whose comments and suggestions have been very useful; in particular for giving us a simplified proof of Theorem \ref{TZ}.

\end{document}